\documentclass[11pt]{article}

\usepackage{amsmath,amsbsy,amsfonts,amsthm,latexsym, 
amsopn,amstext,amsxtra,euscript,amscd,amsthm,xcolor}
          
\newtheorem{lem}{Lemma}
\newtheorem{lemma}[lem]{Lemma}

\newtheorem{thm}{Theorem}
\newtheorem{theorem}[thm]{Theorem}

\def\\{\cr}
\def\({\left(}
\def\){\right)}
\def\[{\left[}
\def\]{\right]}
\def\<{\langle}
\def\>{\rangle}

\begin{document}

\title{On the transcendence of a series related to Sturmian words}

\author{
{\sc Florian~Luca}\\
{School of Mathematics, University of the Witwatersrand}\\
{Private Bag 3, Wits 2050, South Africa}\\
{Research Group in Algebraic Structures and Applications}\\
{King Abdulaziz University, Jeddah, Saudi Arabia}\\
{Max-Planck Institute for Software Systems, Saarbr\"ucken, Germany}\\
{florian.luca@wits.ac.za}
\and
{\sc Jo\"el Ouaknine}\\
{Max Planck Institute for Software Systems}\\
{Saarland Informatics Campus, Saarbr\"ucken, Germany  }\\
{joel@mpi-sws.org}
\and
{\sc James Worrell}\\
{Department of Computer Science}\\
{University of Oxford, Oxford, OX13QD, UK}\\
{jbw@cs.ox.ac.uk}}

\date{\today}

\pagenumbering{arabic}

\maketitle

\begin{abstract}
  Let $b$ be an algebraic number with $|b|>1$ and $\mathcal{H}$ a
  finite set of algebraic numbers.  We study the transcendence of
  numbers of the form $\sum_{n=0}^\infty \frac{a_n}{b^n}$ where
  $a_n \in \mathcal{H}$ for all $n\in\mathbb{N}$.  We assume that the
  sequence $(a_n)_{n=0}^\infty$ is generated by coding the orbit of a
  point under an irrational rotation of the unit circle.  In
  particular, this assumption holds whenever the sequence is Sturmian.
  Our main result shows that, apart from some trivial exceptions, all
  numbers of the above form are transcendental.  We moreover give
  sufficient conditions for a finite set of such numbers to be
  linearly independent over~$\overline{\mathbb{Q}}$.
\end{abstract}

\section{Introduction}
 
For an integer $b\geq 2$, the $b$-ary expansion of a rational number
is eventually periodic.  Over the last few decades, a number of
results have emerged to the effect that an irrational number whose
$b$-ary expansion has low complexity must be transcendental.  For
example, Ferenczi and Maduit~\cite{FM} proved the transcendence of every
irrational number whose $b$-ary expansion is Sturmian.  Recall that
Sturmian words are those with minimal subword complexity among
non-eventually-periodic words.  Indeed, let $p(n)$ denote the number
of distinct length-$n$ factors, then an infinite word is Sturmian if
$p(n)=n+1$ for all $n$, whereas a word is eventually periodic iff it
satisfies $p(n) \leq n$ for some $n$.  The above-mentioned result
of~\cite{FM} was strengthened by Adamczewski, Bugeaud, and
Luca~\cite{ABL}, who showed that if the $b$-ary expansion of an
irrational number has linear subword complexity, i.e., it satisfies
$ \liminf_{n\rightarrow \infty} \frac{p(n)}{n} < \infty \, , $ then
the number must be transcendental.  The approach of~\cite{ABL,FM} can
be refined to derive transcendence measures based on certain
combinatorial characteristics of the $b$-ary expansion of a given
number (see~\cite{AB2,BK}).  In another direction,
Adamczewski~\cite{BorisA} has given lower bounds on the subword
complexity of the $b$-ary expansion of certain transcendental
exponential periods.

Our aim in this paper is to prove transcendence results for numbers of
the form $\sum_{n=0}^\infty \frac{a_n}{b^n}$, where $b$ is complex
algebraic with $|b|>1$, and the $a_n$ are drawn from a finite set
$\mathcal H$ of algebraic numbers.  The assumption that we place on
the sequence $(a_n)_{n=0}^\infty$ is a generalisation of the Sturmian
property---namely that the sequence be the coding of an irrational
rotation on the unit circle.  Roughly speaking, this means that there
is an irrational number $\theta$ and a partition of the unit circle
into finitely many disjoint intervals such that $a_n$ is determined by
the interval containing $n\theta \bmod 1$ for all $n\in\mathbb{N}$.
Morse and Hedlund~\cite{MH} showed that all Sturmian words over a
two-letter alphabet arise as codings of a rotation into two intervals
of respective lengths $\theta$ and $1-\theta$; moreover it is known
that such codings have affine subword complexity function $p(n)=cn+d$
for all sufficiently large~$n$~\cite{BV,ROTE}.  A special case of our
main result is that $\sum_{n=0}^\infty \frac{a_n}{b^n}$ is
transcendental whenever $(a_n)_{n=0}^\infty$ is the coding of a
rotation (and hence whenever $(a_n)_{n=0}^\infty$ is Sturmian).

A key difference between the aforementioned transcendence results
of~\cite{ABL,FM} and the setting of this paper is that our base $b$ is
allowed to be any algebraic number with $|b|>1$ and our set of digits
$\mathcal H$ is allowed to be an arbitrary set of algebraic numbers
rather than $\{0,1,\ldots,b-1\}$.  We note that Adamczewski and
Bugeaud~\cite{AB1} were able to slightly generalise the transcendence
criterion from~\cite{ABL} to accommodate the situation of a base $b>1$
that is a Pisot or Salem number.  The papers~\cite{ABL,AB1,FM} use
$p$-adic versions of Roth's Theorem and the Subspace Theorem.  The
Mahler method has also been used to establish transcendence of numbers
of the form $\sum_{n=0}^\infty \frac{a_n}{b^n}$ for
$(a_n)_{n=0}^\infty$ a non-ultimately periodic automatic sequence and
$b>1$ real algebraic.  An apparent limitation of this approach, as
pointed out by Becker~\cite{Becker}, is that it only appears to work
when $b$ is sufficiently large in terms of $(a_n)_{n=0}^\infty$.

We now introduce the technical setting of our main results.  Let $b$
be complex algebraic with $|b|>1$, and $\theta$ be real
irrational.  Let $\ell\ge 1$ and
\begin{equation}
\label{eq:A}
A=\{r_1,\ldots,r_{\ell}\}\subset (0,1). 
\end{equation}
We assume that $r_1<\cdots<r_{\ell}$ and put $r_0:=0,~r_{\ell+1}:=1$. For a subset ${\mathcal B}\subset [0,1]$  let  
\begin{equation}
\delta_{\mathcal B}(n):={\bf 1}_{{\mathcal N}_{\mathcal B}}\quad {\text{\rm with}}\quad {\mathcal N}_{\mathcal B}:=\{n\ge 0: \{n\theta\}\in {\mathcal B}\}.
\end{equation}
Let ${\bf u}:=(u_0,u_1,\ldots,u_{\ell})\in {\overline{\mathbb Q}}^{\ell+1}$. 
Put
\begin{equation}
\label{eq:T}
T(b,\theta,{A},{\bf u}):=\sum_{n\ge 0}\sum_{i=0}^{\ell} \frac{u_{i}\delta_{[r_i,r_{i+1}]}(n)}{b^n}.
\end{equation}
Our aim is to study conditions under which  $T(b,\theta,A,{\bf u})$ is transcendental. Note that 
\begin{eqnarray*}
T(b,\theta,A,{\bf u}) &=& \sum_{n\ge 0} \sum_{i=0}^{\ell}
                          \frac{u_i(\delta_{[0,r_{i+1}]}(n)-\delta_{[0,r_i]}(n))}{b^n}\\
  &=& \sum_{i=0}^{\ell} \sum_{n\geq 0}
(u_i-u_{i+1})\frac{\delta_{[0,r_{i+1}]}(n)}{b^n}
      \, ,
      \end{eqnarray*}
where $u_{\ell+1}:=0$. The last term in the above sum on the right (when $i=\ell$) is $u_{\ell}\sum_{n\ge 0} 1/b^n=u_{\ell}b/(b-1)$. So, we see that if $u_i=u_{i+1}$ for $i=0,1,\ldots,\ell-1$, 
then $T(b,\theta,A,{\bf u})\in {\overline{{\mathbb Q}}}$. So, we assume that there exists $i\in \{0,1,\ldots,\ell-1\}$ such that $u_i\ne u_{i+1}$. In fact, we may assume that this condition holds for all 
$i=0,\ldots,\ell-1$, for if this condition fails for $i=j\in \{0,\ldots,\ell-1\}$, then we can work with the set $A\backslash \{r_j\}$ (so, we eliminate $r_j$ from $A$). 

In addition, we also assume that 
\begin{equation}
\label{eq:C}
 r_j-r_i\not\in {\mathbb Z}\theta+{\mathbb Z}\quad {\text{for}}\quad 1\le i\ne j\le \ell.
\end{equation}
This does not restrict the generality of our problem. 
Indeed, assume say that $r_j=r_i+v\theta+u$ for some $u,~v\in {\mathbb Z}$. We may suppose that $v\ge 0$ otherwise we swap $r_i$ and $r_j$. Then 
$$
n\theta-r_j\equiv (n-v)\theta-r_i\pmod 1.
$$
Thus, $\{n\theta\}\in [0,r_j]$ if and only if $\{(n-v)\theta\}\in [0,r_i]$, which shows that 
\begin{equation}
\label{eq:C1}
\sum_{n\ge 0}\frac{\delta_{[0,r_j]}(n)}{b^n}=\sum_{n=0}^{v-1} \frac{\delta_{[0,r_j]}(n)}{b^n}+\frac{1}{b^v}\sum_{m\ge 0} \frac{\delta_{[0,r_i]}(m)}{b^{m}}.
\end{equation}
In particular, up to translating $T(b,\theta,{\mathcal A},{\bf u})$ by
an algebraic number and replacing
$u_0,\ldots,u_{j-1},u_{j+1},\ldots,u_{\ell}$ by some linear
combination of themselves with $u_j$ with algebraic coefficients, we
may eliminate $r_j$ out of $A$. Thus, we assume that any two values
among $r_1,\ldots,r_{\ell}$ are incongruent modulo the lattice ${\mathbb Z}\theta+{\mathbb Z}$. 
Note that condition \eqref{eq:C} is  satisfied for example when $r_i\in {\mathbb Q}$ for $i=1,\ldots,\ell$. Indeed, in this case, since $\theta$ is irrational, it follows that if $r_j-r_i\in {\mathbb Z}\theta+{\mathbb Z}$, then 
$r_j-r_i\in {\mathbb Z}$ and since both $r_i$ and $r_j$ are in $(0,1)$, this is impossible. 
Let $b_1,b_2,\ldots,b_k$ be complex numbers. We label them such that
$|b_1|\le  |b_2|\le \cdots \le |b_{\ell}|$ and let $r\in
\{1,\ldots,k\}$ be such that $|b_1|=|b_r|<|b_{r+1}|$.  Then

\begin{theorem}
\label{thm:2}
Let $\theta$ be irrational, $\ell\ge 1$, ${A}$ be the set given by
\eqref{eq:A}, satisfying \eqref{eq:C} and
${\bf u}\in {\overline{\mathbb Q}}^{\ell+1}\backslash \{{\bf 0}\}$
satisfying $u_i\ne u_{i+1}$ for all $i=0,1\ldots,\ell-1$. Assume that
$b_1,\ldots,b_k$ are multiplicatively independent algebraic numbers of
modulus $>1$. Then
$$
1,T(b_1,\theta,{A},{\bf u}),\ldots,T(b_k,\theta,{A},{\bf u})
$$
are linearly independent over ${\overline{\mathbb Q}}$ in the following cases:
\begin{itemize}
\item[(i)] $r=1$;
\item[(ii)] $r=2$, $\theta$ is a ratio of two logarithms of algebraic numbers and $A$ consists of algebraic numbers;
\item[(iii)] $\theta$ has bounded partial quotients.
\end{itemize}
\end{theorem}

As an example application of Theorem~\ref{thm:2}, we 
derive the transcendence of the sum
$$
\sum_{\substack{n\\ \cos(n\theta)>0}} \frac{\cos(n\theta)}{b^n}
$$
in case $\theta$ is real such that $e^{i\theta}$ is algebraic but not
a root of unity.  This answers a question posed in~\cite[Section 4]{FijalkowOPP019} in
relation to a decision problem in control theory.
The application proceeds as follows.
Put $b_1:=be^{i\theta},~b_2:=be^{-i\theta}$ and note that the number above is 
\begin{eqnarray*}
\frac{1}{2}\left(\sum_{\substack{n>0\\ \cos(n\theta)>0}}
  \frac{e^{i\theta n}}{b^n}+\frac{e^{-i\theta n}}{b^n}\right) & = & \frac{1}{2}\sum_{\substack{n>0\\ \cos(n\theta)>0}} \left(\frac{1}{b_1^n}+\frac{1}{b_2^n}\right)\\
& = & \frac{1}{2}T(b_1,\theta_1,{A},{\bf u})+\frac{1}{2}T(b_2,\theta_1,{A},{\bf u}),
\end{eqnarray*}
where $\theta_1:=\theta/(2\pi)=\log(e^{i\theta})/\log(-1)$ is
irrational and a ratio of two logarithms of algebraic numbers,
$\ell=2$, $A=\{1/4,3/4\}\subset {\overline{\mathbb Q}}$,
${\bf u}=(1,0,1)$. Condition \eqref{eq:C} holds since $r_1,r_2$ are
rational. To see that $b_1,~b_2$ are multiplicatively independent,
assume on the contrary that $b_1^x=b_2^y$ for some integers $x,y$ not
both zero. Taking absolute values we get $|b|^x=|b|^y$, so
$x=y$. Thus, the relation $b_1^x=b_2^y$ now simplifies to
$e^{2i x\theta}=1$, a contradiction since $e^{i\theta}$ is not a root
of $1$.  Thus transcendence of
$\displaystyle\sum_{\substack{n>0\\ \cos(n\theta)>0}}\frac{\cos(n\theta)}{b^n}$
follows from Theorem~\ref{thm:2}.  In fact, our result gives more: for
example, we have that
$$
1,\quad \sum_{\substack{n>0\\ \cos(n\theta)>0}}\frac{\cos(n\theta)}{b^n},\quad \sum_{\substack{n>0\\ \sin(n\theta)>0}} \frac{\sin(n\theta)}{b^n}
$$
are linearly independent over ${\overline{\mathbb Q}}$.

We now describe a second consequence of Theorem~\ref{thm:2}.
Let ${\bf v}:=(v_1,\ldots,v_{\ell})\in {\overline{\mathbb Q}}^{\ell}\backslash \{{\bf 0\}}$ and put
\begin{equation}
\label{eq:Sturmian}
{\mathcal S}(b,\theta,A,{\bf v}):=\sum_{n\ge 0} \sum_{i=1}^{\ell}\frac{v_i}{b^{\lfloor n\theta+r_i\rfloor}}.
\end{equation}
Again we assume condition \eqref{eq:C} for if not, say if $r_j-r_i\in {\mathbb Z}\theta+{\mathbb Z}$, then an argument similar to the one from \eqref{eq:C1} shows that
$$
\sum_{n\ge 0} \frac{1}{b^{\lfloor n\theta+r_j\rfloor}}\in {\overline{\mathbb Q}}\sum_{n\ge 0} \frac{1}{b^{\lfloor n\theta+r_i\rfloor}}+{\overline{\mathbb Q}}.
$$
Hence, up to a translate of ${\mathcal S}(b,\theta,A,{\bf v})$ by a
number in ${\overline{\mathbb Q}}$ and up to replacing $v_i$ by a
linear combination of $v_i$ and $v_j$, we can eliminate $r_j$ from
$A$.  Then we have:

\begin{theorem}
\label{thm:22}
Let $\theta$ be irrational, ${A}$ be the set given by \eqref{eq:A} satisfying \eqref{eq:C}, and ${\bf v}\in {\overline{\mathbb Q}}^{\ell}\backslash \{{\bf 0}\}$. Assume that $b_1,\ldots,b_k$ are multiplicatively independent algebraic numbers of modulus $>1$. Then 
$$
1,S(b_1,\theta,{A},{\bf v}),\ldots,S(b_k,\theta,{A},{\bf v})
$$ linearly independent over ${\overline{\mathbb Q}}$ in the following cases:
\begin{itemize}
\item[(i)] $r=1$;
\item[(ii)] $r=2$, $\theta$ is a ratio of two logarithms of algebraic numbers and $A$ consists of algebraic numbers;
\item[(iii)] $\theta$ has bounded partial quotients. 
\end{itemize}
\end{theorem}
The particular case $k=\ell=1$ and $b:=b_1\in {\mathbb Z}$ has been
proved by Adamczewski and Bugeaud in~\cite{AB2}, while the case when $k=1$, $\ell=2$ and $b=b_1\in {\mathbb Z}$ appears in \cite{BH}.

\section{The Subspace Theorem}

Our main tool is the Subspace Theorem which we now recall. For a prime $p$ and $x\in {\mathbb Q}$ we put
$$
|x|_p=p^{-{\text{\rm ord}}_p(x)},
$$
for $x\ne 0$, where ${\text{\rm ord}}_p(x)$ is the exponent of $p$ in
the factorisation of $x$, and $|0|_p=0$. We also put
$|x|_{\infty}:=|x|$ and ${\mathcal M}:=\{\infty\}\cup\{p: p~{\text{\rm
    prime}}\}$.  For all $x\in {\mathbb Q}^*$ we have 
the \emph{product formula} 
$$
\prod_{v\in {\mathcal M}} |x|_v=1 \, .
$$
We extend the $p$-adic valuation to algebraic numbers by putting for $x\in {\overline{\mathbb Q}}$, 
$$
|x|_p=|N_{{\mathbb K}/{\mathbb Q}}(x)|_p^{1/[{\mathbb K}:{\mathbb Q}]}\quad {\text{\rm for}}\quad {\mathbb Q}(x)\subset {\mathbb K}~{\text{\rm and}}\quad [{\mathbb K}:{\mathbb Q}]<\infty.
$$
As is well known, the above formula depends only on $x$ and $p$ and not on the number field ${\mathbb K}$ containing $x$. We will work with linear forms $L({\bf x})\in {\overline{\mathbb Q}}[{\bf x}]$,
where ${\bf x}:=(x_1,\ldots,x_m)$. When specialising $(x_1,\ldots,x_m)$ to a vector in ${\mathbb K}^m$ for some number field ${\mathbb K}$, we will need to work with the infinite valuations of ${\mathbb K}$
extended to ${\overline{\mathbb Q}}$. Assume that the field ${\mathbb K}$ has $r+s$ infinite valuations, where 
$r$ is the number of real ones and $2s$ is the number of complex ones.  Labelling them $\sigma_1,\ldots,\sigma_K$, where $K:=r+s$, they are defined, for $x\in {\mathbb K}$, by
$$
|x|_{\infty_k}:=|x^{(\sigma_k)}|^{\delta_k/[{\mathbb K}:{\mathbb Q}]}\quad {\text{\rm for~all}}\quad k=1,\ldots,K,
$$
where $\delta_k=1$ if ${\mathbb K}^{(\sigma_k)}$ is real and $\delta_k=2$ if ${\mathbb K}^{(\sigma_k)}$ is complex non-real for $k=1,\ldots,K$. We extend these valuations to ${\overline{\mathbb Q}}$ in the same way as we extended the $p$-adic valuations from ${\mathbb Q}$ to ${\overline{\mathbb Q}}$. Namely, if $x\in {\overline{\mathbb Q}}$, we put
$$
|x|_{\infty_k}:=|N_{{\mathbb L}/{\mathbb K}}(x)|^{1/[{\mathbb L}:{\mathbb K}]}_{\infty_k},\quad {\text{\rm where}}\quad {\mathbb K}(x)\subset {\mathbb L}\quad {\text{\rm and}}\quad [{\mathbb L}:{\mathbb Q}]<\infty
$$
for $k=1,\ldots,K$. As in the case of the $p$-adic valuations, the above number depends only on ${\mathbb K}$ and $x$ and does not depend on the number field ${\mathbb L}$ containing ${\mathbb K}(x)$. 
We put ${\mathcal M}_{\mathbb K}$ for the set of all the valuations of ${\mathbb K}$ extended to ${\overline{\mathbb Q}}$, namely
${\mathcal M}_{\mathbb K}:=\{\infty_1,\ldots,\infty_K\}\cup \{p: p~{\text{\rm primes}}\}$. Below is the form of the Subspace Theorem that we use.

\begin{theorem}
\label{thm:ST}
Let ${\mathbb K}$ be a number field and ${\mathcal S}$ be a finite
subset of ${\mathcal M}_{\mathbb K}$ containing all the infinite
valuations on $\mathbb K$. Let $m\ge 2$. For each $v\in {\mathcal S}$, let 
$$
L_{1,v}({\bf x}),\ldots,L_{m,v}({\bf x})
$$
where ${\bf x}:=(x_1,\ldots,x_m)$ be linearly independent linear forms in ${\bf x}$ with coefficients in ${\overline{\mathbb Q}}$. Given $\delta>0$, the set of solutions to 
\begin{equation}
\label{eq:DP}
\prod_{v\in {\mathcal S}}\prod_{i=1}^m |L_{i,v}({\bf x})|_v<\| {\bf x}\|^{-\delta},\qquad {\bf x}\in {\mathcal O}_{\mathbb K}^m
\end{equation}
belongs to finitely many proper subspaces of ${\mathbb K}^m$. Here, 
$$
\|{\bf x}\|:=\max\{|x_i|_v, v\in {\mathcal M}_{{\mathbb K}},~1\le i\le m\}.
$$
\end{theorem}
 Note that since the vector ${\bf x}$ of solutions to inequality \eqref{eq:DP} has algebraic integer components, it follows that $\|{\bf x}\|$ is realised by one of the infinite valuations $|x_i|_v$, $v\in \{\infty_1,\ldots,\infty_K\}$
 of the coordinate $x_i$ for $i\in \{1,\ldots,m\}$ of ${\bf x}$. 
\section{A Transcendence Criterion}

The transcendence of automatic numbers of certain forms has been studied in many papers. See \cite{AB1}, \cite{AB2}, \cite{ABL} for example. Here is the setup.
Let $b\ge 2$ be an integer, $\{a_n\}_{n\ge 0}$ be a sequence with values in a finite set of nonnegative integers say ${\mathcal B}=\{0,1,\ldots,b-1\}$ which is not eventually periodic. 
Consider the infinite word 
$$
{\bf a}:=a_0a_1\ldots a_k\ldots
$$
Assume that there exist two sequences $\{r_n\}_{n\ge 1},~\{s_n\}_{n\ge 1}$ and a number $w>1$ such that:
\begin{itemize}
\item[(i)] $r_n/s_n=O(1)$;
\item[(ii)] The sequence $\{s_n\}_{n\ge 1}$ tends to infinity; 
\item[(iii)] ${\bf a}=U_n{\underbrace{V_n\ldots V_n}_{w~{\text{\rm times}}}}\ldots$, where $U_n,~V_n$ have lengths $r_n$ and $s_n$, respectively (here, we mean that the first $r_n+\lfloor w s_n\rfloor$ letters of ${\bf a}$ and $U_nV_nV_n\ldots$ coincide). 
\end{itemize}
Then the number
$$
S_b({\bf a}):=\sum_{n\ge 1} \frac{a_n}{b^n}
$$
is transcendental. This is the main result in \cite{ABL} (see also \cite{AB1} and \cite{AB2}). A few comments are in order. For example, how important is it that the set of values of ${\bf a}$ 
is $[0,b-1]\cap {\mathbb Z}$? Can it be any finite set of algebraic numbers?
Can one replace the condition $b$ being an integer by the weaker condition 
that $b$ is algebraic with $|b|>1$? In this paper, we address these questions. 

In the rest of this section we reproduce the proof from \cite{ABL}. In the next sections we suitably modify it and pay attention to the eventual obstructions for the method to go through. In the last section we show 
that our sequences fulfil all the criteria that we introduce along the way and we get
the announced results. Put $\alpha:=S_b({\bf a})$. The proof uses (iii) and introduces
\begin{equation}
\label{eq:1111}
\alpha^{(n)}:=\sum_{k\ge 1} \frac{a_k^{(n)}}{b^n},
\end{equation}
where ${\bf a}^{(n)}:=a_1^{(n)}a_2^{(n)}\ldots a_k^{(n)}\ldots$ is the approximant $U_nV_nV_nV_n\cdots$ of ${\bf a}$. More precisely the numbers $a_k^{(n)}$ appearing in $\alpha^{(n)}$ are given by 
$$a_k^{(n)}:=\left\{\begin{matrix} a_k & {\text{\rm for}} & k\le r_n+ws_n;\\ 
a_{k+s_n}^{(n)} & {\text{\rm for}} & k\ge r_n.\end{matrix}\right. 
$$
Certainly,  since $\alpha^{(n)}$ has the compact formula 
$$
\alpha^{(n)}=\frac{p_n}{b^{r_n}(b^{s_n}-1)}
$$
for some $p_n\in {\mathbb Z}$, we see that (i) leads to 
$$
|\alpha-\alpha^{(n)}|<\frac{1}{b^{r_n+ws_n}}.
$$
This in turn leads to 
\begin{equation}
\label{eq:smallform}
|\alpha b^{r_n+s_n}-\alpha b^{r_n}-p_n|<\frac{1}{b^{(w-1)s_n}}.
\end{equation}
Assuming $\alpha$ is algebraic, the above is a linear form  in three variables 
$$
L(x_1,x_2,x_3):=\alpha x_1-\alpha x_2-x_3,
$$
with algebraic coefficients which is ``small" in the Archimedean valuation $\infty$ infinitely often at points $(x_1,x_2,x_3):=(b^{r_n+s_n},b^{r_n},p_n)$ of which two are powers of $b$, in particular composed only of primes dividing $b$. Condition (i) controls the height of the above integer vector $(x_1,x_2,x_3)$.  That is, it says that $\|{\bf x}\|\ll |x_i|^{\eta}$ holds for  all $i=1,2,3$ with a suitable $\eta>0$. The number $\eta$ can be taken to be $C_1/(C_1+1)$, where $C_1>r_n/s_n$ holds for all $n\ge 1$.  Condition (ii) ensures that there are infinitely many solutions to the above inequality \eqref{eq:smallform}. 
An immediate application of the Subspace Theorem (with ${\mathbb K}={\mathbb Q}$ and ${\mathcal S}=\{\infty\}\cup \{p: p\mid b\}$) gives that infinitely many of those points must satisfy a linear equation. We give these details in subsequent sections. But we already have a natural candidate for the linear equation 
namely $L(b^{r_n+s_n},b^{r_n},p_n)=0$. One shows that in fact, only this linear form can vanish infinitely often (other potential candidates of fixed linear forms vanishing on $(b^{r_n+s_n},b^{r_n},p_n)$ 
give only finitely many possibilities for $n$), but this leads to $\alpha$ being rational. Since $a_n$ has values in $\{0,1,\ldots,b-1\}$, the series $S_b({\bf a})$ is in fact the base $b$ expansion of $\alpha$ and one now invokes the elementary criterion that $\alpha$ is rational only if $\{a_n\}_{n\ge 0}$ is eventually periodic which is not the case. This gives the desired contradiction. 

\section{A New Transcendence Criterion}
\label{sec:variations}

This bird's eye view of the proof of the main result in \cite{ABL} shows that if one wants to make progress one needs to get better at two things:
\begin{itemize}
\item[(1)] get better (``smaller") expressions like \eqref{eq:smallform}. 
\item[(2)] replace the requirement that  $\{a_n\}_{n\ge 0}$ take
  values in $\{0,1,\ldots,b-1\}$ by a combinatorial condition on ${\bf
    a}$ that allows the
$a_n$ to have values in any finite set of algebraic numbers. 
\end{itemize}
We start by considering transcendence of a single number $S_b({\bf a})$.
We assume that $b$ is algebraic with $|b|>1$. As
suggested by (2) above, we shall assume that the set of values of
${\bf a}$ is a finite set of algebraic numbers denoted ${\mathcal
  H}$. We put ${\mathbb K}:={\mathbb Q}(b)$ and assume it has degree
$D$. Up to multiplying through by a common denominator of the numbers
in ${\mathcal H}$, we assume that they are all algebraic integers and
we let $H$ be an upper bound for the house (largest absolute values of
the conjugates) of any of these numbers.  We keep the sequences
$\{r_n\}_{n\ge 1}$ and $\{s_n\}_{n\ge 1}$ satisfying (i) and (ii) and
we assume additionally that they are strictly increasing.  We carry
over also the definition of the ultimately periodic approximant
${\bf a}^{(n)}$ of $\bf a$.  As before, we assume that $C_1>r_n/s_n$ for all
$n\ge 1$.  Along the way, we will find other constants
$C_2,C_3,\ldots$. They all depend on our data
${\bf a},b,b_1,\ldots,b_k$, but not on
$w$. If we want to write something depending on $w$, we will emphasise
the dependence by writing $C(w)$, or $O_w(1)$.

We replace (iii) by the following requirement:
\begin{itemize}
\item[(iii.1)] ${\text{\rm BPP}}$:
  For each integer $w>1$ there exists $n_w$ such that for all $n \geq
  n_w$ there exists $t_n$ such that,  putting $m:=r_n+ws_n$, we have
  $$\left\{ j \in [0,m) : a_j \neq a_j^{(n)} \right\} =
  \left( \{i_1(n),\ldots,i_{t_n}(n)\} + \mathbb{Z}_{\geq 0}s_n\right)
  \cap [0,m) \, . $$ We assume that the $i_{\ell}(n)$ are distinct
  modulo $s_n$ and write $I_w(n)$ for the union of all arithmetic
  progressions on the right-hand side above.  We further require that
  for $j=1,2,\ldots,w-1$,
$$
\#(\{i_1(n),\ldots,i_{t_n}(n)\}\cap [r_n+js_n,r_n+(j+1)s_n)])=O(1).
$$
\item[(iii.2)] ${\text{\rm EGP}}$: If $t_n\ge 2$, then $i_{\ell}(n)-i_{\ell-1}(n)$ tends to infinity with $n$ for $\ell=2,3,\ldots,t_n$.
\item[(iii.3)] ${\text{\rm LPP}}$: There exists a function $f_0:{\mathbb N}\mapsto {\mathbb N}$ such that $f_0(m)$ tends to infinity with $m$ and for $\ell=1,\ldots,t_n$ and $m\in {\mathbb Z}_{\ge 0}$, we have
$$
a_{i_{\ell}(n)+m s_n}=a_{i_{\ell}(n)+(m+1)s_n},\quad {\text{\rm for~all}}\quad 0\le m\le f_0(s_{n+1}).
$$ 
\end{itemize}

We use BPP for ``Bounded Progression Property", EGP for ``Expanding Gaps Property" and LPP for ``Long Pattern Property". Note that (iii.1) implies that $t_n=O(w)$.

\subsection{Using condition (iii.1)}

Let us see what is the advantage of the above condition (iii.1). We follow the method from \cite{ABL}. Let any $w>1$ be arbitrarily large but fixed. We will see how large we need it later. Let $n>n_w$ 
and let $\alpha^{(n)}$ be given by \eqref{eq:1111} with the same
definition of $a_k^{(n)}$. Condition (iii.1) implies
$$
\left|\alpha-\alpha^{(n)}-\sum_{\ell\in I_w} \frac{c_{\ell}}{b^{i_{\ell}(n)}(b^{s_n}-1)}\right|<\frac{2H}{|b|^{r_n+ws_n}},
$$
where $c_\ell:=a_{i_{\ell}(n)}-a_{i_{\ell}(n)-s_n}$ for $\ell=1,2,\ldots,t_n$. Note that since $a_n\in {\mathcal H}$ it follows that the $c_\ell$'s have values in the finite set ${\mathcal H}-{\mathcal H}$ of algebraic integers. In particular, the house of $c_{\ell}$ is at most $2H$ for $\ell=1,\ldots,t_n$. Since $t_n=O(w)$ and $w$ is fixed, we may assume that $t_n=t$ is fixed. Since $t$ is fixed and the $c_{\ell}$'s take values in a finite set, we may assume that the $c_{\ell}$'s are fixed for $\ell=1,\ldots,t$. Clearly, there are finitely many choices for $(t,c_1,\ldots,c_t)$. Replacing $\alpha^{(n)}$ by its formula we get
\begin{equation}
\label{eq:ttuu}
\left|\alpha-\frac{p_n}{b^{r_n}(b^{s_n}-1)}-\sum_{\ell=1}^t \frac{c_{\ell}}{b^{i_{\ell(n)}}(b^{s_n}-1)}\right|<\frac{2H}{|b|^{r_n+ws_n}}.
\end{equation}
Multiplying across by $b^{r_n}(b^{s_n}-1)$, we get
\begin{equation}
\label{eq:main1}
 \left|\alpha b^{r_n+s_n}-\alpha b^{r_n}-p_n -\sum_{\ell=1}^t c_{\ell}b^{r_n-i_\ell(n)}\right|
 <  \frac{2H}{|b|^{(w-1)s_n}}.
\end{equation}
The left--hand side in \eqref{eq:main1} above is a linear form in $t+3$ indeterminates
\begin{equation}
\label{eq:L}
L(u_1,u_2,y,z_1,z_2,\ldots,z_t):=\alpha u_1-\alpha u_2-y-\sum_{\ell=1}^t c_{i_\ell}z_\ell.
\end{equation}
Assuming $\alpha\in {\overline{\mathbb Q}}$, the above form has coefficients which are algebraic numbers in the extension ${\mathbb L}:={\mathbb K}(\alpha,{\mathcal H})$. Here is our intermediate result. 
\begin{lemma}
\label{lem:important1}
Assume that conditions (i), (ii) and (iii.1) hold, that $b$ is
algebraic, $|b|>1$ and that $\alpha:=S_b({\bf a})$ is also algebraic.
Then there exists a constant $C_2:=C_2(b,{\bf a})$ such that for
$w>C_2$ and all $n>n_w$ the following holds: There exist $C_3(w)$, a
finite set $E$ of numbers in ${\mathbb L}$ containing $0$ (which
depends on $w$), $\ell\in \{1,2,\ldots,t_n\}$ with $i_{\ell}(n)>ws_n-C_3(w)$ in $I_w(n)$, such that
the left-hand side of \eqref{eq:ttuu} equals
\begin{equation}
\label{eq:indep1}
\alpha-\frac{p_n}{b^{r_n}(b^{s_n}-1)}-\sum_{u=1}^{t_n} \frac{c_u}{b^{i_{u}(n)}(b^{s_n}-1)}=\frac{e}{b^{i_{\ell}(n)}(b^{s_n}-1)}\quad {\text{for~some}}\quad e\in E.
\end{equation}
If $I_{w}(n)=\emptyset$, then we understand that in the above equation the number $e$ is zero. 
\end{lemma}

\begin{proof}
For simplicity we write $t:=t_n$. Let $B$ be a common denominator for $b$ and $1/b$. 
That is, $B$ is a positive integer such that both $bB$ and $B/b$ are algebraic integers. Assume $w>C_1+2$.
Multiplying \eqref{eq:main1} by $B^{ws_n}$, we get
\begin{eqnarray}
\label{eq:interm}
&& \left|\alpha B^{ws_n}b^{r_n+s_n}-\alpha B^{ws_n} b^{r_n}-B^{ws_n}p_n-\sum_{\ell=1}^t c_{\ell} B^{ws_n} b^{r_n-i_{\ell}(n)}\right|\nonumber \\
& < & \frac{2H|B|^{ws_n}}{|b|^{(w-1)s_n}}.
\end{eqnarray}
The left--hand side is $L({\bf x})$, where ${\bf x}:=(x_1,\ldots,x_{t+3})$ and $L$ is given by \eqref{eq:L}. We label the coordinates of ${\bf x}$ as $(u_1,u_2,y,z_1,\ldots,z_t)$, where 
\begin{equation}
\label{eq:x1x2}
(u_1,u_2,y):=\left(B^{ws_n} b^{r_n+s_n}, B^{ws_n} b^{r_n},B^{ws_n}p_n\right),
\end{equation}
and 
\begin{equation}
\label{eq:z}
z_{\ell}:=B^{ws_n} b^{r_n-i_{\ell}(n)}\quad {\text{\rm for}}\quad \ell=1,\ldots,t.
\end{equation}
These vectors have algebraic integer components in ${\mathcal O}_{\mathbb K}$. Indeed the first two are clear, the last $t$ are so since 
$B$ is a common denominator of $b$ and $1/b$ and $r_n+ws_n\ge r_n+i_{\ell}(n) s_n$ for all $\ell=1,\ldots,t$. The only one that is in doubt is $x_3$ but since
$$
p_n=\sum_{i=0}^{r_n+s_n} a_ib^{r_n+s_n-i},
$$
it follows that 
$$
p_nB^{r_n+s_n}=\sum_{i=0}^{r_n+s_n} a_i(bB)^{r_n+s_n-i}B^i\in {\mathcal O}_{\mathbb K}.
$$
So, in fact $x_3$ is an algebraic integer which is a multiple of $B^{(w-1)s_n-r_n}$. Recall that $w>C_1+2$, so the exponent of $B$ above is positive. 
We consider the set ${\mathcal S}$ consisting of all the infinite places of ${\mathbb K}$ (where we adopt the convention that the regular absolute value of ${\mathbb K}$ is denoted $\infty_1$), 
and the primes $p$ dividing $B$ or such 
that $|b^{(\sigma)}|_p\ne 1$ for some $\sigma\in {\text{\rm Gal}}({\mathbb K}/{\mathbb Q})$. We extend these valuations to ${\mathbb L}:={\mathbb K}(\alpha,{\mathcal H})$. 
The system of linear forms is
$$
L_{i,\nu}({\bf x}):=x_i\quad {\text{\rm for~all}}\quad (i,\nu)\in \{1,\ldots,t+3\}\cup {\mathcal S}\backslash \{(3,\infty_1)\},
$$
and 
$$
L_{3,\infty_1}({\bf x}):=L({\bf x}).
$$
For each $\nu\in {\mathcal S}$ the $t+3$ forms are linearly independent. We compute the double product
$$
\prod_{\nu\in {\mathcal S}}\prod_{i=1}^{t+3} |L_{i,\nu}({\bf x})|_{\nu},
$$
where the coordinates of ${\bf x}$ are given by \eqref{eq:x1x2} (the first three) and \eqref{eq:z} (the last $t$).  By the product formula and the fact that all coordinates except the third one are 
${\mathcal S}$-units, the subproducts corresponding to any fixed $i\ne 3$ in the set $\{1,2,\ldots,t+3\}$ (and all valuations $\nu$ in ${\mathcal S}$) is $1$. For the third coordinate, the product over the 
finite places is at least $B^{-((w-1)s_n-r_n)}$ since this coordinate is an algebraic integer divisible by $B^{(w-1)s_n-r_n}$. For the infinite place $\infty_1$, we get that
$$
|L_{3,\infty_1}({\bf x})|_{\infty_1}\ll \left(\frac{B^{ws_n}}{|b|^{(w-1)s_n}}\right)^{1/\delta_1},
$$
where $\delta_1\in \{1/D,2/D\}$. Finally, for the infinite places $\infty_k$ with $k>1$, we get that
$$
|L_{3,\infty_k}({\bf x})|_{\infty_k}=|x_3|_{\infty_k}\ll s_n |B|^{ws_n\delta_k} |B_1|^{(r_n+s_n)\delta_k},
$$
where $\delta_k\in \{1/D,2/D\}$, the implied constant can be taken to be $2|H|(1+C_1)$, and we put $B_1:=\max\{B|b^{(\sigma)}|:\sigma\in {\text{\rm Gal}}({\mathbb K}/{\mathbb Q})\}$. Thus,
$$
\prod_{\nu\in {\mathcal S}}\prod_{i=1}^{t+3} |L_{i,\nu}({\bf x})|_{\nu}\ll \frac{s_n B^{s_n+r_n} B_1^{r_n+s_n}}{|b|^{(w-1)s_n/D}}\ll  s_n\left(\frac{B_1^{2(C_1+1)}}{|b|^{(w-1)/D}}\right)^{s_n}.
$$
Taking $w$ such that 
$$
w>4D(C_1+1)\frac{\log B_1}{\log |b|}+1,
$$
we get that 
$$
\prod_{\nu\in {\mathcal S}}\prod_{i=1}^{t+3} |L_{i,\nu}({\bf x})|_{\nu}\ll \frac{s_n}{|b|^{\eta_1 s_n}}\ll \frac{1}{\|{\bf x}\|^{\eta_2}}\quad {\text{\rm provided}}\quad w>\frac{4D}{\log |b|}+1,
$$
where we can take 
$$
\eta_1:=\frac{w-1}{2D},\quad \eta_2:=\frac{\delta_1\log |b|}{2(C_1+w)\log B}=\frac{(w-1)\log |b|}{4D(C_1+w)\log B_1}
$$ 
and $w$ is large enough such that  $|b|^{\eta_1/2}>2$ (so we can use the inequality $|b|^{\eta_1 s_n/2}>2^{s_n}>s_n$ which holds for all $n\ge 1$). 
An interesting feature of $\eta_2$ is that it is bounded from below by the quantity $\log |b|/(8D\log B_1)$ for sufficiently large $w$.  Further, the height of 
our points is at least as large as $(r_n+s_n)\log |b|$, so they are ``large points".
Now the Subspace Theorem tells us that such points ${\bf x}$ lie in finitely many proper subspaces.  There are bounds on the number of such subspaces. 
For that, one looks at ``small points" and ``large points". The large points are the ones whose height exceeds the height of the form $L$. Our points have this property for 
all $n>n_0$, where $n_0:=n_0(b,{\bf a})$. The number of subspaces containing ``large points" is bounded by $\exp(O(t(w)^2))=\exp(O(w^2))$, where the constant implied by the above $O$ depends on $b$ and ${\bf a}$. Thus, there exists a finite set of linear 
equations of the form
\begin{equation}
\label{eq:linear}
L_1({\bf x})=\sum_{i=1}^{t+3} d_i^{(\lambda)} x_i=0\quad {\text{\rm for}}\quad \lambda=1,2,\ldots,T(w),
\end{equation}
with coefficients $d_i^{(\lambda)} \in \mathbb{Q}(b)$ for
$i=1,\ldots,t+3$ depending on $w$, not all zero, such that each of our
points ${\bf x}$ satisfies one of the above equations.  And all we
have to do is to show that any of the equations given at
\eqref{eq:linear} has $O_w(1)$ solutions $n$ except if it it
equivalent (proportional) to the equation $L({\bf x})=0$ or it is
equivalent to one of the additional equations \eqref{eq:indep1}
permitted by the lemma.  So, let's do it. We fix one such $\lambda$
and then omit the dependence on the superscript $\lambda$ of the
coefficients $d_i$'s. Assume that ${\bf x}$ satisfies equation
\eqref{eq:linear} with some linear form $L_1({\bf x})$ not parallel to
$L({\bf x})$. Let $j\in \{1,\ldots,t+3\}$ be such that $x_j$ appears
in $L_1({\bf x})$ with coefficient $d_j\ne 0$. Assume that it appears
in $L({\bf x})$ with coefficient $e_j$. Replace $L({\bf x})$ by
$L({\bf x})-(e_j/d_j)L_1({\bf x})$, which is not the zero form.  The
value of the left--hand side of \eqref{eq:main1} is unchanged but it
is now a linear form in $t+2$ variables
$x_1,\ldots,x_{j-1},x_{j+1},\ldots,x_{t+3}$. By the Subspace Theorem
again there is some nonzero linear equation among the variables which
cannot be $L({\bf x})-(e_j/d_j)L_1({\bf x})=0$ since that would imply
$L({\bf x})=0$, which we assume not to hold. Thus, we can eliminate
one more variable. Going in this way, we replace $L({\bf x})$ by
$$
L({\bf x})-\mu_1L_1({\bf x})-\mu_2L_2({\bf x})-\cdots,
$$ 
where at each step $L_i({\bf x})=0$ and each $L_i({\bf x})$ has at most $t+3-i$ indeterminates appearing in it with a nonzero coefficient. At the end of the day, we end up with a form in one variable. 
If that variable is $x_3$, we then get
that 
$$
|p_n|\ll_w \frac{1}{|b|^{(w-1)s_n}},
$$
which implies
\begin{equation}
\label{eq:pn}
\frac{|p_n|}{|b|^{r_n}|b^{s_n}-1|}\ll_w \frac{1}{|b|^{r_n+ws_n}}.
\end{equation}
The implied constant depends on $w$ since it comes from one of the above subspaces. If this will happen for arbitrarily large values of $w$ and infinitely many $n$, we would get that $p_n=0$. In particular, 
$\alpha=0$. Returning to our inequalities \eqref{eq:main1}, we get that the left--hand side of \eqref{eq:main1} is a linear form in ${\mathcal S}$-units, which is small. If it has at least two ${\mathcal S}$-units in it, then we can write one of finitely many ${\mathcal S}$-unit equations and use each one of them to eliminate another variable. At the end of the day, we get either $L({\bf x})=0$, which we assumed not to hold, or if  $L({\bf x})\ne 0$, there is one variable among $x_1,x_3,z_{\ell}$'s which survives. If it is among $x_1,x_2$ (so $\alpha\ne 0$), we get
$$
|b|^{r_n}\ll_w \frac{1}{b^{(w-1)s_n}},
$$
which  a bound on $n$. This was assuming $\alpha\ne 0$, for if $\alpha=0$, then $x_1,x_2$ did not appear at all. Thus,  it remains to analyse the case when the variable is among the $z_{\ell}$'s. We then get
$$
L({\bf x})=eb^{r_n-i_{\ell}(n)},
$$
which leads to the desired conclusion by dividing across by $b^{r_n}(b^{s_n}-1)$. Note that in this case since $|L({\bf x})|\ll b^{-r_n-ws_n}$, we get $i_{\ell}(n)\ge ws_n-C_3(w)$, as claimed.  
\end{proof}

To see that this is the right formulation, let us see a
multidimensional version.  We start with algebraic numbers $b_1,\ldots,b_k$ 
which are multiplicatively independent and of absolute values larger than $1$. We would like to show that 
$$
1,S_{b_1}({\bf a}),\ldots,S_{b_k}({\bf a})
$$
are linearly independent over ${\overline{\mathbb Q}}$ under certain conditions. 
To show this we can set up the same machine as in \eqref{eq:ttuu}.
We assume that $k\ge 2$ and that
\begin{equation}
\label{eq:19}
\lambda_0+\lambda_1S_{b_1}({\bf a})+\cdots +\lambda_kS_{b_k}({\bf a})=0
\end{equation}
for some $\lambda_0,\ldots,\lambda_k$ algebraic numbers not all zero. 
We write estimates \eqref{eq:ttuu} for $\alpha_i=S_{b_i}({\bf a})$ for $i=1,\ldots,k$ and take an appropriate linear combination of them to get
$$
\left|\lambda_0+\sum_{j=1}^k\lambda_j\frac{p_n(b_j)}{b_j^{r_n}(b_{j}^{s_n}-1)}+\sum_{j=1}^k\sum_{\ell=1}^{t} \frac{d_{\ell}^{(j)}}{b_{j}^{i_\ell(n)}(b_j^{s_n}-1)}\right|\ll 
\frac{1}{|b|^{r_n+ws_n}}.
$$   
Here, $|b|:=\min\{|b_j|: j=1,\ldots,k\}$ and $d_{\ell}^{(j)}:=\lambda_j c_{\ell}$ are algebraic numbers for $\ell=1,\ldots,t$ and $j=1,\ldots,k$ which are linear combinations of the original $c_{\ell}$ for $\ell=1,\ldots,t$ with 
the coefficients $\lambda_1,\ldots,\lambda_k$. They belong to the finite set $\sum_{j=1}^k \lambda_{j} {\mathcal H}$. Here is the next result. 

\begin{lemma}
\label{lem:important2}
Assume that ${\bf a}$ satisfies (i), (ii), (iii.1). Assume that
$b_1,\ldots,b_k$ are algebraic, $|b_j|>1$ for $j=1,\ldots,k$ and
that there is a nontrivial linear relation with algebraic coefficients
among $1,S_{b_1}({\bf a}),\ldots,S_{b_k}({\bf a})$:
\begin{equation}
\label{eq:lindep}
\lambda_0+\sum_{j=1}^k\lambda_{j}S_{b_{j}}({\bf a})=0.
\end{equation}
Then there exists a constant $C_4:=C({\bf a},b_1,\ldots,b_k,\lambda_0,\lambda_1,\ldots,\lambda_k)$ such that for each $w>C_4$ and $n>n_w$, there exists a finite set $E$ depending on $w$ such that for all but finitely many $n$, one of the equations 
\begin{equation}
\label{eq:indep2}
\lambda_0+\sum_{j=1}^k \lambda_{j} \frac{p_n(b_{j})}{b_j^{r_n}(b_\ell^{s_n}-1)}+\sum_{j=1}^k\sum_{u=1}^{t_n} \frac{d_{i_u(n)}^{(j)}}{b_j^{i_u(n)}}=\frac{e\prod_{m=1}^k b_m^{\delta_j s_n}}{b_j^{{i_{\ell}(n)}}(b_1^{s_n}-1)\cdots (b_k^{s_n}-1)}
\end{equation}
holds for some $j=1,\ldots,k$ and $\ell \in \{1,\ldots,t_n\}$, where if $e\ne 0$ then $\delta_i\in \{0,1\}$ for $i=1,\ldots,m$ with $\delta_j=0$. In addition, if $e\ne 0$, then $i_{\ell}(n)>C_5 ws_n$ for some constant $C_5$ provided $s_n$ is large enough.  
\end{lemma}

\begin{proof}
We work with ${\mathbb K}:={\mathbb Q}(b_1,\ldots,b_k)$ and as before 
we denote its degree by $D$. Then we get
\begin{eqnarray}
\label{eq:main3}
&& \left|\lambda_0\prod_{j=1}^kb_j^{r_n}(b_\ell^{s_n}-1)+\sum_{j=1}^k \lambda_{j} \left(\prod_{\substack{m\ne j\\ 1\le m\le k}} b_m^{r_n}(b_m^{s_n}-1)\right) p_n(b_{j})\right.\nonumber\\
& - & \left.\sum_{j=1}^k b_j^{r_n} \sum_{\ell=1}^{t} d_{\ell}^{(j)}b_j^{-i_{\ell}(n)} \prod_{\substack{1\le m\le k\\ m\ne j}}^k (b_m^{s_n}-1)\right|\nonumber\\
& \ll & \frac{1}{|b|^{r_n+ws_n-C_6(r_n+s_n)}}\ll \frac{1}{|b|^{(w-C_7)s_n}},
\end{eqnarray}
where $|b|:=\min\{|b_j|: 1\le j\le k\}$. Here, we can take $C_6:=(\sum_{j=1}^k \log |b_j|)/\log |b|)$, and $C_7:=C_6(C_1+1)$. In the left, we expand all parenthesis and have a linear form in 
$2^k+k2^{k-1}+2^{k-1}kt$ variables. Putting 
$$
M:=\prod_{j=1}^k b_{j}^{r_n+s_n},
$$
our variables are 
$$
M\left(\prod_{j\in {\mathcal T}} b_{j}^{s_n}\right)^{-1},\quad T\subseteq \{1,\ldots,k\},
$$
$$
M/(b_{j}^{r_n+s_n})\times \left(\prod_{m \in {\mathcal T}_{j}} b_{m}^{s_n}\right)^{-1}p_n(b_{j}),\quad 1\le j\le k,~T_{j}\subseteq \{1,\ldots,k\}\backslash \{j\},
$$
and
$$
Mb_j^{-i_{\ell}(n)}\left(\prod_{m\in {\mathcal T}_j} b_{m}^{s_n}\right)^{-1},\quad T_j\subseteq \{1,\ldots,k\}\backslash \{j\},~1\le j\le k,~1\le \ell\le t.
$$
Letting $B$ be a common denominator for $b_j,1/b_{j}$ for all $j=1,\ldots,k$, we multiply both sides of \eqref{eq:main3} by $B^{kr_n+ks_n}$. The finite set of valuations ${\mathcal S}$ 
consists of the infinite ones of ${\mathbb K}$ together with the finite ones $p$ such that either $p$ is a factor of $B$ or $|b_j^{(\sigma)}|_p\ne 1$ for some $j=1,\ldots,k$ and some $\sigma\in {\text{\rm Gal}}({\mathbb K}/{\mathbb Q})$. 
We extend these valuations  to ${\mathbb L}:={\mathbb K}(\lambda_0,\ldots,\lambda_k,{\mathcal H})$. 
All coordinates except for $k2^{k-1}$ of them (the ones involving the expressions $p_n(b_j)$ for $j=1,\ldots,k$) are ${\mathcal S}$-units. We take the same system of forms 
namely $L_{i,\nu}({\bf x}):=x_i$ except for one form corresponding to the infinite place  $\nu_1$ which is the embedding corresponding to $b$, and the index $i:=2^k+1$ (first indeterminate containing one of $p_n(b_1)$), where we take it to be $L_{2^k+1,\nu_1}({\bf x})=L({\bf x})$ where this form is the one from the left--hand side of 
\eqref{eq:main3} after expanding all the parenthesis. A similar calculation 
as in Lemma~\ref{lem:important1} shows that with this system of forms and for the above points ${\bf x}$, we have, 
$$
\prod_{\nu\in {\mathcal S}}\prod_{i=1}^{2^k+k2^{k-1}+2^kkj} |L_{\nu,i}({\bf x})|_{\nu}\ll_{H,k,\Lambda} \frac{(s_nB_1^{2(C_1+1)s_n})^{k2^{k-1}}}{|b|^{(w-C_7)s_n/D}}.
$$
The implied constant depends on $k,H$ and $\Lambda$, where this last parameter is an upper bound on the houses (largest conjugate) of the algebraic integers $\lambda_0,\ldots,\lambda_k$ and 
$$
B_1:=\max\{B|b_j^{(\sigma)}|: 1\le j\le k, \sigma\in {\text{\rm Gal}}({\mathbb K}/{\mathbb Q})\}.
$$    
Taking $w>4D(C_1+1)2^{k-1}k\log B_1/\log |b|+C_7$, it follows that the factor involving $B_1$ can be absorbed 
into the denominator at the cost of halving the exponent of $|b|$, namely from $(w-C_7)/(2D)$ replacing it by $(w-C_7)/(4D)$. Assuming further that $(w-C_7)/(8D)>k2^{k-1}/\log |b|$, we may in fact 
also absorb the power of $s_n$ from the denominator of the right--hand side into the numerator at the cost of replacing the exponent of $|B|$ from $(w-C_7)/(4D)$ by $(w-C_7)/(8D)$. Thus,
$$
\prod_{\nu\in {\mathcal S}}\prod_{i=1}^{2^k+k2^{k-1}+2^kkj} |L_{\nu,i}({\bf x})|_{\nu}\ll_{H,k,\Lambda}\frac{1}{|b|^{(w-C_7)/(8D)}}\ll \frac{1}{\| {\bf x}\|^{\eta}},
$$
where we can take 
$$
\eta:=\frac{(w-C_7)\log |b|}{8D((C_1+w)\log B_1}.
$$
As in the $1$-dimensional case, $\eta$ is bounded from below by $\log |b|/(9D\log B_1$) once $w$ is sufficiently large. The conclusion of the 
Subspace Theorem is that ${\bf x}$ satisfies one of finitely many linear relations. The number of relations is exponential in the square of the number of variables so it is $\exp(O(w^2))$, where now 
the constant implied by $O$ depends on $k,~{\bf a}$ and $b_1,\ldots,b_k$ and this holds for all $n>n_0$, where now $n_0$ depends on $|b|,~B_1,~H,~\Lambda$ only. We need to exploit these relations.
In the $1$-dimensional case we succeeded in proving that all but finitely many $n$ satisfy the linear relation given by imposing that the left--hand side of \eqref{eq:main1} is either $0$ or one of the involved variables
arising from the exponents $i_{\ell}(n)$ for some $\ell\in \{1,\ldots,t\}$. We will prove that the same holds for \eqref{eq:main3}. Let's see the details. Assume $L({\bf x})\ne 0$
infinitely often. We pick a linear form $L_1({\bf x})$ such that $L_1({\bf x})=0$. Clearly, $L_1({\bf x})$ is not parallel to $L({\bf x})$. 
We pick an indeterminate $x_{j}$ which appears in $L_1({\bf x})$ with nonzero coefficient and replace $L({\bf x})$ by
$L({\bf x})-\mu_1L_1({\bf x})$ for some suitable nonzero coefficient $\mu_1$ in order to eliminate $x_j$, obtaining in  such a way a ``small" linear form in fewer variables (at least the variable $x_j$ is no longer present). Note that this new small linear form is not zero since otherwise $L({\bf x})=0$, which is something we assume not to hold. 
We continue in this way at each stage creating a  ``small" linear form in fewer variables which is a linear combination of $L({\bf x})$ 
with other linear forms encountered along the way, and all except for $L({\bf x})$ vanish at our vector ${\bf x}$. Hence, the new linear form in fewer variables does not vanish at our ${\bf x}$. 
At the end we end up with the last linear form in one variable $L'({\bf x})=x_i$ being small. If this is one of the small indeterminates containing $p_n(b_j)$ for some $j=1,\ldots,k$, (and $\lambda_j\ne 0$ otherwise these coordinates did not appear to begin with) then we argue as before that 
$p_n(b_j)=0$ for all but finitely many $n$ satisfying this equation. So, again we get fewer variables and continue. So, let us assume that $x_i$ is an ${\mathcal S}$-unit indeterminate. If it is 
one of the first $2^k$ small ${\mathcal S}$-units indeterminate, we get a contradiction for large $n$ unless $\lambda_0=0$ so these variables did not appear to begin with.
The final case is when the variable is one of the large ${\mathcal S}$-unit indeterminates involving some $i_{\ell}(n)$ in the exponent. Then we get
$$
L({\bf x})=eb_j^{-i_{\ell}(n)} (b_1\cdots b_k)^{r_n}\prod_{m}^k b_m^{\delta_m s_n},\quad  \delta_m\in \{0,1\},~1\le m\le k,~\delta_j=0.
$$
Further, $e\in E$, where $E$ is finite (depends on $w$) and there are finitely many choices for $(j,\delta_1,\ldots,\delta_k)$. The desired equation follows by dividing both sides above by 
$\prod_{i=1}^k b_i^{r_n}(b_i^{s_n}-1)$. Note further that since $|L({\bf x})|\ll |b|^{-(w-C_7)s_n}$, we get that
\begin{equation}
\label{eq:10}
|b_j|^{r_n-i_{\ell}(n)}\prod_{\substack{1\le m\le k\\ m\ne j}} ^k |b_m|^{r_n+\delta_m s_n} \ll_w \frac{1}{|b|^{(w-C_7)s_n}}\quad {\text{\rm where}}\quad \delta_m\in \{0,1\}
\end{equation}
for all $m=1,\ldots,j$. This gives $i_{\ell}(n)\log |b_j|\ge (w-C_8)s_n\log |b|+O_w(1)$, where $C_8:=C_7+k(C_1+1)\log |b^*|/\log |b|$ and $b^*:=\max\{|b_j|, 1\le j\le k\}$. 
In particular, taking $C_4:=2C_8$, $w>C_4$, and $s_n>n_w$, we see that there is indeed $C_{5}$, which can be taken to be $C_5:=C_4\log |b|/(2\log |b^*|)$, such that if $w>C_{4}$ and $s_n>n_w$, then $i_{\ell}(n)>C_{5} w s_n$, which is what we wanted. 
\end{proof}


\section{The Easy Case}
We will use Lemmas~\ref{lem:important1} and~\ref{lem:important2} to
obtain linear independence properties of the numbers $S_b({\bf a})$
over $\overline{\mathbb{Q}}$.  For this we augment
conditions (iii.1)--(iii.3) of Section~\ref{sec:variations}.
There are several different cases.  We start with the easiest one.

\begin{itemize}
\item[(iv.1)] For any $\varepsilon\in (0,1)$ there exist arbitrarily
  large positive integers $w$ such that for infinitely many $n$, the
  interval $[r_n+\varepsilon ws_n,r_n+ws_n]$ does not contain any
  $i_{\ell}(n)$.
\end{itemize}

Taking $\varepsilon$ sufficiently small (smaller than $1/2$ in the
case of Lemma \ref{lem:important1} and smaller than $C_6$ in the case
of Lemma \ref{lem:important2}), choosing a large $w$ which is
convenient for us and satisfies (iv.1), then for infinitely many large
$n$ (for example larger than $2C_4(w))$) condition (iv.1) will apply
to show that in the right--hand side of \eqref{eq:indep1} and
\eqref{eq:indep2} we have the number $e=0$.  For the 1-dimensional
case we get that
$$
\alpha=\frac{p_n}{b^{r_n}(b^{s_n}-1)}+\sum_{\ell=1}^{t} \frac{c_{\ell}}{b^{i_{\ell}(n)}(b^{s_n}-1)}
$$
holds. The right--hand side encodes the values of $a_n$ only up to $r_n+ws_n$. Comparing it with the expression for $\alpha$, we get that 
\begin{equation}
\label{eq:ell}
\sum_{\substack{p>r_n+ws_n\\ p\not\equiv i_{\ell}(n)\pmod {s_n},~\ell=1,2,\ldots,t}} \frac{a_p-a_{\overline{p}}}{b^p}+O\left(\sum_{p>f_0(s_{n+1}) s_n}\frac{1}{b^p}\right)=0,
\end{equation}
where we put ${\overline{p}}$ for the unique index in $ [r_n+1,r_n+s_n]$ such that $p\equiv {\overline{p}}\pmod {s_n}$. The second sum encodes the difference of values of $a_{i_\ell(n)+ms_n}$ 
and $a_{i_{\ell}(n)}$ for $m>f_0(s_{n+1})s_n$. And it remains to decide if relation \eqref{eq:ell} 
can happen infinitely often. 

Here is an easy to check condition under which it cannot happen infinitely often:
\begin{itemize}
\item[(v.1)] ${\text{\rm ETGP}}$: Letting $\kappa_1(n)<\kappa_2(n)$ be
  the first two indices $p$ which are larger than $r_n+ws_n$ such that
  $\kappa_i(n)\not\equiv i_{\ell}(n)$ for all $\ell=1,\ldots,t$ and
  both $i=1,2$, $a_{p}\neq a_{\overline{p}}$ for both
  $p\in \{\kappa_1(n),\kappa_2(n)\}$ and
  $a_{\kappa_1(n)}\neq a_{\kappa_2(n)}$, suppose that
  $\kappa_2(n)=o(f_0(s_{n+1})s_n)$ as $n\to\infty$ and that both
  $\kappa_1(n)$ and $\kappa_2(n)-\kappa_1(n)$ tend to infinity with
  $n$.
\end{itemize}
Then relation \eqref{eq:ell} implies that 
\begin{eqnarray}
\label{eq:ending}
1 & \ll & |a_{\kappa_1(n)}-a_{{\overline{\kappa_1(n)}}}|\nonumber\\
& = & \left|\sum_{p\ge \kappa_2(n)} \frac{a_{p}-a_{\overline{p}}}{b^{p-\kappa_1(n)}}+O\left(\sum_{p\ge f_0(s_{n+1})s_n} \frac{1}{b^{p-\kappa_1(n)}}\right)\right|\nonumber\\
& \ll & \frac{1}{|b|^{\kappa_2(n)-\kappa_1(n)}},
\end{eqnarray}
which yields a contradiction for values of $n$ such that $\kappa_2(n)-\kappa_1(n)$ is sufficiently large.  We call (v.1) ETGP for ``Expanding Tail Gaps Property". 
Here is what we have proved.

\begin{theorem} 
\label{thm:u1}
Assume that ${\bf a}$ satisfies (i), (ii), (iii.1)--(iii.3), (iv.1) and (v.1). Then for every algebraic $b$ with $|b|>1$ the number $S_b({\bf a})$ is transcendental.
\end{theorem} 
In particular, $\alpha\ne 0$ in the case of Lemma \ref{lem:important1} under (iv.1). As such $p_n(b_j)\ne 0$ for any $j$ and $n$ large enough under (iv.1). 

The multidimensional version works equally well. Namely, we write equation \eqref{eq:indep2} with $e=0$ for a suitable large $w$ and infinitely many $s_n$.  We subtract \eqref{eq:indep2} from \eqref{eq:lindep} and get the analogue of \eqref{eq:ending} 
\begin{equation}
\label{eq:kappas1}
\left|(a_{\kappa_1(n)}-a_{\overline{\kappa_1(n)}})\left(\sum_{\ell=1}^k\lambda_{\ell} b_{\ell}^{-\kappa_1(n)}\right)\right|\ll \sum_{t\ge \kappa_2(n)} \frac{|a_{t}-a_{\overline{t}}|}{|b|^{t}}\ll \frac{1}{|b|^{\kappa_2(n)}}.
\end{equation}
In case $|b_1|,\ldots,|b_k|$ are all distinct, the above inequality is impossible for large $n$ by (iv.1). Namely, we leave on the left--hand side only $b_1$ which realises the minimum absolute values among them and put the rest of the terms in the right--hand side and apply the previous argument. This will take care of (i) of Theorem \ref{thm:2} assuming (iv.1) holds. Let us see how to deal with (ii) of Theorem \ref{thm:2} assuming (iv.1) holds.

So, assume that $r=2$. In this case, dividing by $\lambda_1$ we get
$$
\left|(a_{\kappa_1(n)}-a_{\overline{\kappa_1(n)}})\left(1+\frac{\lambda_2}{\lambda_1}\left(\frac{b_1}{b_2}\right)^{\kappa_1(n)}\right)\right|\ll \frac{1}{|b_1|^{\kappa_2(n)-\kappa_1(n)}}.
$$
and in the left-hand side the second factor can be small. However, by
lower bounds for linear forms in logarithms, we have that
\begin{equation}
\label{eq:linearform}
\left|1+\frac{\lambda_2}{\lambda_1}\left(\frac{b_1}{b_2}\right)^{\kappa_1(n)}\right|\gg \frac{1}{\kappa_1(n)^{C_{9}}}
\end{equation}
for some constants $C_{9}>0$ depending on $b_1,b_2,\lambda_1,\lambda_2$. Thus, the above inequalities \eqref{eq:kappas1} and \eqref{eq:linearform} give us
$$
\kappa_2(n)-\kappa_1(n)\le C_{9} \log \kappa_1(n)+O(1).
$$
Again we can increase the value of $C_{9}$ (say replace $C_{9}$ by $C_{10}=2C_{9}$) and assume that $n$ is large in order to omit the additive $O(1)$ term in the right--hand side above. Thus,
\begin{equation}
\label{eq:kappas}
\kappa_2(n)-\kappa_1(n)\le C_{10} \log \kappa_1(n).
\end{equation}

So, we formulate the following criterion.

\begin{itemize}
\item[(v.2)] Assume $r=2$ and that for any $C_{10}\ge 1$, inequality \eqref{eq:kappas} holds only for finitely many $n$.
\end{itemize}

So, (v.1) and (v.2) are enough to deal with statements (i) and (ii) of our Theorem \ref{thm:2} assuming (iv.1) is satisfied. Part (iii) is dealt with in the next section. 

\section{The Harder Case}

\begin{itemize}
\item[(iv.2)] Assume that there exists $C_{11}$ and $w_0$ such that
  for all $w>w_0$, the interval $[r_n+C_{11}ws_n,r_n+ws_n]$ contains at least one $i_{\ell}(n)$. 
\end{itemize}

Since $\kappa_1(n)$ is the first $i_{\ell}(n)$ larger than $r_n+ws_n$ (for some new $w'$), we get that $\kappa_1(n)\in [r_n+ws_n,r_n+C_{12}ws_n]$, where $C_{12}=C_{11}^{-1}$. Further,
$\kappa_2(n)<r_n+C_{12}^2 ws_n$, etc. So, let $\kappa_1(n)<\kappa_2(n)<\kappa_j(n)<\ldots$ be the points starting the progressions of changes after $r_n+ws_n$ as we increase $w$. As we have seen,
$\kappa_j(n)<r_n+C_{12}^j w s_n$. In particular, $\kappa_j(n)/s_n=O_j(w)$. In addition we want

\begin{itemize}
\item[(v.3)] There exist a function $f_1(w)$ tending to infinity such that for large $w$ and any fixed $j$ there exists $i$ such that 
$$
\kappa_i(n)<\kappa_i(n+1)<\cdots<\kappa_{i+j}(n)<r_n+w^2s_n,
$$
and furthermore $\kappa_{i+j}(n)-\kappa_i(n)\ll_j \kappa_{i+j}(n)/f_1(w)$.
\item[(v.4)] Furthermore, there exists a number $L$ such that if $j>L$, there are $\ell'>\ell\in \{i,i+1,\ldots,i+j\}$, with $k_{\ell'}(n)-k_\ell(n)\gg_j \kappa_{\ell'}(n)/f_2(w)$ for some function $f_2(w)$ tending to infinity.  
\end{itemize} 

Let us now finish. We return to our equations \eqref{eq:indep1} and to its multidimensional analogue \eqref{eq:indep2}. We let $K$ be larger than the bound $O(1)$ from (iii.1). We let $j:=3K+3$. Part (v.4) above gives us a string of $\kappa_i(n),\ldots,\kappa_{i+j}(n)$ which are close together. Note that all of them are  of the form $i_{\ell}(n)$ for some $\ell$ once we change $w$ to $w^2$. By (iii.1), the interval $[\kappa_i(n),\kappa_{i+j}(n)]$ will contain at least three multiples of $s_n$. Let any of the middle one (so not the first or last) be of the form $u+w_1s_n$ for some $u\le r_n$. We change $w$ to $w_1$. That is we cut--off our problem at $r_n+w_1s_n$. 
Everything (number of variables, coefficients $c_i$, etc.) are in finitely many configurations controlled by $w^2$. Let $i_0$ be such that $\kappa_{i+i_0}(n)\le r_n+w_1s_n<\kappa_{i+i_0+1}(n)$. 
As far as this $r_n+w_1s_n$ is concerned, the numbers $\kappa_{i+\ell}(n)$ for $\ell\le i_0$ are $i_{\ell'}(n)$ for some $\ell'\le t_n=O(w_1)$ and $\kappa_{i+i_0+1}(n),~\kappa_{i+i_0+2}(n)$ and larger ones are $\kappa_1(n),~\kappa_2(n)$, etc. Since the relative ratios $\kappa_{i+\ell}(n)/\kappa_{i+\ell-1}(n)$ tend to $1$ (can be made arbitrarily close to $1$ by (v.3) and by choosing a sufficiently large $w$), in the left--hand sides we only keep $b_1,\ldots,b_r$ among the numbers which realize the minimum of the absolute values and put the rest in the other side. The same goes for the eventual variable involving $e$ 
(if $e$ is nonzero), which we expand in series using the products of $1/(b_m^{s_n}-1)$ for $m=1,\ldots,k$ and keep only the first term. And we get
$$
\left| \sum_{m=1}^r \frac{c_m}{b_m^{\kappa_{i+i_0+1}(n)}}-\frac{e}{b_j^{i_{\ell}(n)}\prod_{m=1}^j b_m^{(1-\delta_m)s_n}}\right|\ll \frac{1}{|b_1|^{(1+\delta_w)\kappa_{i+i_0}(n)}}
$$
for some positive $\delta_{w_1}$ which depends on $w_1$, but is otherwise bounded from below in terms of $w$. In the left, we have a linear form in ${\mathcal S}$-units which is small. Thus, there are finitely many linear equations in these variables. 
We may assume that they are non-degenerate. Any equation involving two of the $b_j$'s will give only finitely many values for $\kappa_{i+i_0+1}(n)$ since the $b_j$'s are multiplicatively independent. 
So, $r=1$ and $e$ is nonzero. It now follows that the only possibility is that identically $c_1/b_1^{\kappa_{i+i_0+1}(n)}$ equals the unknown involving $e$ with finitely many exceptions. This implies that $j=1$, and $\delta_m=0$ for $m\ne j$, so $\kappa_{i+i_0+1}(n)=i_{\ell}(n)+O_w(1)$. Since $\kappa_{i+i_0+1}(n)-i_{\ell}(n)$ tends to infinity, we get that $O_w(1)$ is not present so 
$\kappa_{i+i_0+1}(n)=i_{\ell}(n)$, but this is wrong since $i_{\ell}(n)<r_n+ws_n\le r_n+w_1s_n<\kappa_{i+i_0+1}(n)$, namely these two changes were sitting on opposite sides of $r_n+w_1s_n$.
This argument shows that under conditions (iv.2), (v.3) and (v.4) we may assume that $e=0$ in equations \eqref{eq:indep1} and \eqref{eq:indep2}. Now the argument from the previous section takes care of the case $r=1$ so we assume that $r\ge 2$. We need to deal with \eqref{eq:ell}. We may move all the $b$'s such that $|b_i|>|b|$ (so $i\ge r+1$), to the right--hand side and leave only the ones with the same absolute value in the left. In a first step we go to inequality \eqref{eq:kappas1}. By the Subspace theorem, for any $\varepsilon>0$, if $n$ is large enough the left--hand side of \eqref{eq:kappas1} exceeds $|b|^{(1+\varepsilon)\kappa_1(n)}$. Thus, 
if $\kappa_2(n)>(1+2\varepsilon)\kappa_1(n)$ for all large  $n$, then we are done. If not, it means that $\kappa_2(n)$ is very close to $\kappa_1(n)$. We shall assume that $\kappa_2(n)/\kappa_1(n)<\log |b_{r+1}|/\log |b_1|$ (in case $r<k$). Then we can also incorporate the tails corresponding to $\kappa_2(n)$ in the left, so as to write it as 
\begin{eqnarray*}
&& \left|(a_{\kappa_1(n)}-a_{\overline{\kappa_1(n)}})\left(\sum_{\ell=1}^r \frac{\lambda_\ell}{b_\ell^{\kappa_1(n)}}\right)+(a_{\kappa_2(n)}-a_{\overline{\kappa_2(n)}})\left(\sum_{\ell=1}^r \frac{\lambda_\ell}{b_\ell^{\kappa_2(n)}}\right)
\right|\\
& \ll & \max\left\{\frac{1}{b^{\kappa_3(n)}},\frac{1}{|b_{r+1}|^{\kappa_1(n)}}\right\}.
\end{eqnarray*}
The left--hand side again by the Subspace Theorem exceeds $|b|^{(1+\varepsilon)\kappa_1(n)}$ unless there are some degeneracies (zero subsums in the left--hand side). Since the $b_i$'s are multiplicatively independent, for large $n$ the only degeneracies can come from some $i$ and from the terms $b_i^{\kappa_1(n)}$ and $b_i^{\kappa_2(n)}$, and would lead to the conclusion that $b_i^{\kappa_2(n)-\kappa_1(n)}$ is in a fixed finite set, and this is impossible for large 
$n$ since $\kappa_2(n)-\kappa_1(n)$ tends to infinity by (iii.2). So, assuming again that $\varepsilon$ is small enough, say $1+\varepsilon<\log |b_{r+1}|/\log |b_1|$, we get that the only possibility is that also 
$\kappa_3(n)\le (1+\varepsilon)\kappa_2(n)+O(1)$. In particular that $\kappa_3(n)\le (1+2\varepsilon)\kappa_1(n)$ for large $n$. Then  we incorporate $b_i^{\kappa_3(n)}$  for $i=1,2,\ldots, r$ to the left as well. Going in this way, we get that $\kappa_2(n)/\kappa_1(n)$, later $\kappa_3(n)/\kappa_1(n)$ and so on are all smaller than $1+\varepsilon$ where $\varepsilon$ can be chosen as small as we want. However, by (v.4) we know that by the time we get to $\kappa_{L+1}(n)$ this can no longer be the case since $\kappa_{L+1}(n)/\kappa_1(n)>1+1/f_2(w)$ for some function $f_2(w)$. So, if this patters continues for $L$ steps we then get $1/f_2(w)<\varepsilon$, so if we choose an 
$\varepsilon$ smaller than $1/f_2(w)$, we reach a contradiction. Let us record what we proved.

\begin{theorem}
\label{thm:u2}
Assume that ${\bf a}$ satisfies (i), (ii), (iii.1)--(iii.3), (iv.2), and (v.1), (v.3), (v.4). Then for every algebraic numbers $b_1,\ldots,b_k$ of absolute values larger than $1$ and multiplicatively independent we have that 
$$
1,S_{b_1}({\bf a}),\ldots,S_{b_k}({\bf a})
$$ 
are linearly independent over ${\overline{\mathbb Q}}$. 
\end{theorem}
Note now that the case (i), namely $r=1$ of Theorem \ref{thm:2} is covered. Namely, either (iv.1) holds or (iv.2) hold and in either case the statement follows from Theorems \ref{thm:u1} and Theorem \ref{thm:u2}, respectively. Theorem \ref{thm:u2} also covers case (iii). The case (ii), namely $r=2$ and extra conditions on $\theta$ and $A$, is covered by the following theorem.

\begin{theorem}
\label{thm:u3}
Assume that ${\bf a}$ satisfies (i), (ii), (iii.1)--(iii.3), (v.1) and (v.2). Then for every algebraic numbers $b_1,\ldots,b_k$ of absolute values larger than $1$ and multiplicatively independent satisfying $r=2$, we have that 
$$
1,S_{b_1}({\bf a}),\ldots,S_{b_k}({\bf a})
$$ 
are linearly independent over ${\overline{\mathbb Q}}$. 
\end{theorem}
The combination of these theorems covers our first Theorem \ref{thm:2} in light of the properties from the next section. We comment about Theorem \ref{thm:22} after the next section.

\section{Continued Fractions}

This section is independent of the previous ones, so we can relabel our variables. We assume now that $k$ is the cardinality of the set of boundary points $A$.

Let $\theta\in {\mathbb R}\backslash {\mathbb Q}$, $k\ge 1$,
$r_1,\ldots,r_k$ be distinct numbers in $(0,1)$.  We further set
$r_0:=0$ and $r_{k+1}:=1$.
Let $\theta=:[a_0,a_1,\ldots,a_m,\ldots]$ be the continued fraction expansion of $\theta$ 
and $\{p_m/q_m\}_{m\ge 0}$ be the sequence of its convergents. Let $\delta_i(n)$ be the characteristic function of the set $\{n: \{n\theta\}\in [r_i,r_{i+1}]\}$ for $i=0,\ldots,k$ and 
$$
\delta(n)=(\delta_0(n),\ldots,\delta_k(n)):{\mathbb N}\mapsto \{0,1\}^{k+1}.
$$
For each integer $w\ge 1$ and $N\ge 1$, let 
$$
I_w(N)=\{1\le M<q_N: \delta(M+\ell_1q_N)\ne \delta(M+\ell_2 q_N)~{\text{\rm for~some}}~0\le \ell_1<\ell_2\le w\}.
$$
We are interested in the structure of the elements in $I_w(N)$ for all $w$ especially when it comes to verifying (iii.1)--(iii.3) as well as the rest of the conditions (iv) and (v).  
We assume without loss of generality that $w$ is an integer.  We take $r_n=s_n=q_n$ to be the denominator of the $n$th convergent to $\theta$.

We may assume that $\theta$ is positive;  if it is not positive then we replace $\theta$ by $-\theta$ and $r_i$  by $1-r_i$ for $i=1,\ldots,k$. We put $\eta:=\min\{r_{i+1}-r_i: i=0,1,\ldots,k\}$
for the smallest of the lengths of ${\mathcal I}_i=[r_{i+1},r_i]$.
Recall that
$$
q_m\theta-p_m=\frac{(-1)^{m-1}}{q_m\theta_{m+1}+q_{m-1}},
$$
where $\theta_{m+1}:=[a_{m+1},a_{m+2},\ldots]$. In particular,
$$
\{q_m\theta\}\equiv \frac{(-1)^{m-1}}{q_m\theta_{m+1}+q_{m-1}}\pmod 1.
$$
Thus, assuming $q_m>2(w+1)/\eta$, we have that 
$$
\{q_m \ell\theta\}\equiv \frac{(-1)^{m-1} \ell}{q_m\theta_{m+1}+q_{m-1}}\pmod 1.
$$
For $\ell\le w$, the numbers shown on the right are in $(0,1)$ and either they are all smaller than $\eta/2$ or all within $\eta/2$ of $1$ depending on the parity of $m$. This means that for a fixed positive integer $M$ the numbers 
$$
\{M\theta\}, \{M+q_N\theta\},\ldots,\{(M+w q_N)\theta\}
$$
are respectively congruent modulo $1$ with 
$$
\{M\theta\}, \{M\theta\}+\{q_N\theta\},\ldots,\{M\theta\}+\{wq_N\theta\}\pmod 1.
$$
Let $i\in \{1,\ldots,k\}$ be such that $\{M\theta\}\in [r_i,r_{i+1}]$. If we have the disequality $\delta(M+\ell_1q_N)\ne \delta(M+\ell_2q_N)$ it follows that $\{(M+wq_N)\theta\}\not\in [r_i,r_{i+1}]$. Thus, there exists a unique minimal
$\ell\in \{0,1,\ldots,w-1\}$, such that $\{M+\ell q_N\}\in [r_i,r_{i+1}]$ and $\{(M+(\ell+1)q_N\}\not\in [r_i,r_{i+1}]$. This shows that $\{(M+(\ell+1)q_N \theta\}$ is either in $[r_{i+1},r_{i+2}]$ or 
in $[r_{i-1},r_i]$ where the indices are taken modulo $k$ (so if $i=k$, then $[r_{k},r_{k+1}]$ means $[r_{k},1]\cup [0,r_1]$). We assume that this is $r_{i+1}$. 
Further, the distance from one of $\{(M+\ell q_N)\theta\}$ and $\{(M+(\ell+1)q_N)\theta\}$ to $r_{i+1}$ is at most $\| q_N\theta\|/2$ (since the sum of the distances from 
$\{(M+\ell q_N)\theta\}$ and $\{(M+(\ell+1)q_N\}$ to $r_{i+1}$ is
exactly $\|q_N\theta\|$). We assume that it is the distance from $\{(M+\ell q_N)\theta\}$ to $r_{i+1}$ that is smaller than or equal to $\|q_N\theta\|/2$.
In particular,
$$
\delta((M+jq_N)\theta)=({\underbrace{0,0,\ldots,0}_{x~{\text{\rm times}}}},1,0,\ldots),
$$
where $x=i$ for $j=0,1,\ldots,\ell$ and $x=i+1$ if $j=\ell+1,\ldots,w$. 
We write this as 
$$
\{(M+\ell q_N)\theta\}=r_i+\zeta_{M,\ell,N},
$$
where $|\zeta_{M,\ell,N}|\le \|q_N\theta\|/2$. Thus,
$$
(M+\ell q_N)\theta=p_{M,\ell,N}+r_i+\zeta_{M,\ell,N},\quad p_{M,\ell,N}\in {\mathbb Z},
$$
or 
\begin{equation}
\label{eq:1}
(M\theta-T_M)-r_i=\frac{\ell (-1)^{N-1}}{q_N\theta_{N+1}+q_{N-1}}+\zeta_{M,\ell,N},
\end{equation}
where $T_m:=p_{M,\ell,i}-p_N$. We need an upper bound for the number
of positive integers $M<q_N$ arising from such a representation for
some $\ell\le w$ and $\zeta_{M,\ell,N}$ a real number of absolute
value at most $\| q_N\theta\|/2$. Let us note that $\ell$ and $i$
determine $M$ in at most two ways. Indeed assume that
$(M,T_M),~(M',T_M'),~(M'',T_M'')$ are all solutions of an equation like \eqref{eq:1} for the same $\ell$ and $i$  
and some different numbers $\zeta_{M,\ell,N},\zeta_{M',\ell,N},~\zeta_{M'',\ell,N}$. Two of the $\zeta$'s will have the same sign. Assume they are $\zeta_{M,\ell,N}$ and $\zeta_{M',N,\ell}$. Then 
$$
|(M-M')\theta-(T_M-T_{M'})|=|\zeta_{M,\ell,N}-\zeta_{M',\ell,N}|<|\zeta_{M,\ell,N}|+|\zeta_{M',\ell,N'}|\le \|\theta q_N\|.
$$
If $M\ne M'$, then $|M-M'|<q_N$, so by known facts about continued fractions the left--hand side is at least $|q_N\theta-p_N|=\|\theta q_N\|$ a contradiction.  
Thus, $M=M'$ and then $T_M=T_{M'}$. So, here is our first result.

\begin{lemma}
  The constant mentioned in (iii.1) can be taken to be equal to twice
  the number of elements of $A$.
\end{lemma}

This confirms (iii.1). For (iii.2) note that $i_{\ell}(n)$ is of the form $u+\ell s_{n}$, with $\ell\le w$ and $u\le s_n$ such that for some $i$, we have
\begin{equation}
\label{eq:122}
|\{i_{\ell}(n)\theta\}-r_i|=O\left(\frac{1}{s_{n+1}}\right).
\end{equation}
The implied constant above can be taken to be $1$. Further, at $m:=i_{\ell}(n)$, $\{m\theta\}$ just changed from say having been in the interval $[r_{i-1},r_{i}]$ to being into the interval $[r_{i},r_{i+1}]$. But then in order to change it again by adding multiples of $s_n$ to it, we need to add at least
$\gg \eta\times s_{n+1}$ of such multiples. This confirms (iii.3) with $f_0(s)\gg \eta s$. Finally, if $t_n\ge 2$ and $\ell\ge 2$, then there exists $j$ such that 
\begin{equation}
\label{eq:122prime}
|\{i_{\ell-1}(n)\theta\}-r_{j}|=O\left(\frac{1}{s_{n+1}}\right).
\end{equation}
Relations \eqref{eq:122} and \eqref{eq:122prime} show that
$$
|(i_{\ell}(n)-i_{\ell-1}(n))\theta-T-(r_i-r_j)|=O\left(\frac{1}{s_{n+1}}\right)\quad {\text{\rm holds~with~some}}\quad T\in {\mathbb Z}.
$$
For $i=j$, the above relation shows that $\| (i_{\ell}(n)-i_{\ell-1}(n))\theta\|\ll s_{n+1}^{-1}$, which shows that $i_{\ell}(n)-i_{\ell-1}(n)$ tends to infinity. If $i\ne j$, then $(i_{\ell}(n)-i_{\ell-1}(n))\theta$ is close (within 
$O(1/s_{n+1})$) from one of the finitely many numbers $r_i-r_j\pmod 1$ for $i\ne j\in \{1,\ldots,k\}$. This shows that $i_{\ell}(n)-i_{\ell-1}(n)$ tends to infinity unless $(i_{\ell}(n)-i_{\ell-1}(n))\theta$ is exactly one of the above numbers modulo $1$, but this is impossible because of condition \eqref{eq:C}. Hence, we have just confirmed the following.

\begin{lemma}
Conditions (iii.1)--(iii.3) are satisfied.
\end{lemma}

We next check condition (v.1) is always satisfied and that condition (v.2) is satisfied when $\theta$ is a ratio of logarithms of algebraic numbers and $A$ consists of algebraic numbers. 

Let $w$ be given, and take $r_1$. We pick $\varepsilon$ to be small (we will figure it out how small we need it later). Take the intervals $[r_1-2\varepsilon,r_1-\varepsilon]$ and $[r_1+\varepsilon,r_1+2\varepsilon]$. 
Take the numbers $q_n$ for large $n$. By uniform distribution, for large $n$, there are at least $2kw+2$ numbers $\kappa$ such that $\{\kappa \theta\}$ is in 
$[r_1-2\varepsilon,r_1-\varepsilon]$ and the same is true for $[r_1+\varepsilon,r_1+2\varepsilon]$. Assume of course that $q_n$ is larger than any one of these $4kw+4$ numbers. Now take each one of these and start adding multiples of 
$q_n$ to them. Since adding one extra $q_n$ changes the distance to the nearest integer of that resulting multiple of $\theta$ by $O(1/q_{n+1})$, it follows that after about $W_{\varepsilon}:=\lfloor C_{13} \varepsilon q_{n+1}\rfloor+1$ steps  (where we can take $C_{13}$ to be equal to $6$) all the first numbers passed on the side larger than $r_1$ if $n$ was odd, and on the side smaller than $r_1$ if $n$ was even. Since those numbers remain congruent to 
the initial numbers we have chosen modulo $q_n$, it follows that once they change not all of them can be in $I_w(n)$ because this set has only at most $2kw$ progressions. In fact at least two of them are outside. This shows that $\kappa_2(n)=O(\varepsilon q_n q_{n+1})$
and since this was true for $\varepsilon$, we get that $\kappa_2(n)=o(q_n q_{n+1})=o(f_0(s_{n+1})s_n)$ as $n$ tends to infinity. To see that $\kappa_2(n)-\kappa_1(n)$ tends to infinity, we distinguish two cases, namely whether they caused a change with respect to the same $r_i$ or not. In the first case, note that $\kappa_2(n)-\kappa_1(n)$ is of the form 
$k_2-k_1$, where both $\{k_2\theta\}$ and $\{k_1\theta\}$ were at most $\varepsilon$ apart, plus some integer of the form $(\ell_2-\ell_1)q_n$, where $\max\{\ell_1,\ell_2\}=O(\varepsilon q_{n+1})$. Thus, 
$\|(\kappa_2(n)-\kappa_1(n)\theta\|=O(\varepsilon)$. Since $\varepsilon$ was arbitrary, this proves the statement. In the second case, $k_2\theta$ was within $O(\varepsilon)$ of ${\mathbb Z}+r_i$ and 
$k_1\theta$ was within $O(\varepsilon)$ of ${\mathbb Z}+r_j$ for some $i\ne j$. Thus, $\|(k_2-k_1)\theta-(r_j-r_i)\|=O(\varepsilon)$. Since $\varepsilon$ is arbitrary and the left--hand side cannot be zero by condition \eqref{eq:C}, we conclude that $k_2-k_1$ tends to infinity. Thus, (v.1) is verified.

Let us now verify the condition (v.2) in the case $\theta$ is a ratio of two logarithms of algebraic numbers and $A$ consists of algebraic numbers. By the Erd\H os-T\'uran-Koksma inequality (see \cite{HerLuc} for a fun application of this) and Baker's method, we have that $\varepsilon$ can be chosen on a scale of $O(q_n^{-\delta})$ for some small constant $\delta$. Further, $q_{n+1}=q_n^{O(1)}$ by Baker's method. The previous argument shows that $\kappa_2(n)=q_n^{O(1)}$ and the ending of our argument shows that 
\begin{equation}
\label{eq:kappa1kappa2}
\left|(\kappa_2-\kappa_1)\theta-T-(r_i-r_{i'})\right|=O\left(\frac{1}{q_n^{\delta}}\right)\qquad {\text{\rm for~some}}\quad T\in {\mathbb Z},
\end{equation}
where $i,i'\in \{1,\ldots,k\}$.  If $i=i'$, the left--hand side is not zero. If $i\ne i'$, the left--hand side is not zero by condition \eqref{eq:C}. If $\kappa_2-\kappa_1=O(\log \kappa_2)$, then $\kappa_2-\kappa_1=O(\log q_n)$. By linear forms in logarithms in the left--hand side of the inequality \eqref{eq:kappa1kappa2} (here is where we need that $A$ consists of algebraic numbers), we get that the left--hand side of \eqref{eq:kappa1kappa2} 
is $\gg 1/(\log q_n)^{O(1)}$, which gives
$$
q_n^{\delta}\ll (\log q_n)^{O(1)},
$$
so $q_n=O(1)$, a contradiction. Let us summarise what we have proved.

\begin{lemma}
Condition (v.1) holds. Condition (v.2) holds when $r=2$, $\theta$ is a ratio of two logarithms of algebraic numbers and $A$ consists of algebraic numbers.  
\end{lemma}

\subsection{The easy case}

Next let us show that condition (iv.1) is satisfied when $\{a_m\}_{m\ge 1}$ is unbounded. Take $C$ to be any constant. Look at elements $i_{\ell}(n)$ in the interval $[r_n+ws_n,r_n+Cws_n]$. They are of the form 
$i_{\ell}(n)=u+mq_n$ for some $m\in [w,Cw]$, and 
$$
\{u+m q_n)\theta\}=\{u\theta\}+O\left(\frac{Cw}{a_{n+1}q_n}\right).
$$
The implied constant above can be taken to be $1$. So, if $i_{\ell}(n)$ exists of this form then $\{u\theta\}$ is very close to some $r_i$, and so $\|(u+ m q_n)\theta-r_i\|=O(Cw/(a_{n+1}q_n)$. For a fixed $i$, we saw that $u$ can take at most two values. But assuming $a_{n+1}\gg Cw\eta^{-1}$, there cannot be another value $m'\ne m$ in $[w,Cw]$ such that $u+m'q_n=i_{\ell'}(n)$.  Thus, the interval $[r_n+ws_n,r_n+Cws_n]$ contains at most $2k$ integers of the form $i_{\ell}(n)$. Putting $C':=C^{1/(2k+1)}$, one of the intervals 
$[r_n+C'^jw, r_n+C'^{j+1}w]$ for $j=0,1,\ldots,k$ does not contain any number of the form $i_{\ell}(n)$. Since $C$ can be made arbitrarily large, so can $C'$. So, we have (iv.1), which we record.

\begin{lemma}
If $\{a_m\}_{m\ge 0}$ is unbounded, then we have (iv.1).
\end{lemma}

\subsection{The hard case}

Assume that $\{a_m\}_{m\ge 0}$ is bounded. We need to verify that (iv.2), (v.3) and (v.4) hold. By a result of Khintchine (see \cite{Khi}), there are infinitely many $u$ such that $\| u\theta-r_1\|=O(1/u)$. The constant 
in $O$ can be taken to be $1/{\sqrt{5}}$. Taking such an $u$ and the least $n$ such that $q_n>u$, we have $q_n\asymp u$, so $\| u\theta-r_1\|=O(1/u)=O(1/q_n)$. Since $\|\ell q_n \theta\|\gg \ell/q_n$ 
for any fixed $\ell$ and large $n$, we get that if we add sufficient large multiples of $q_n$ to $u$ we will find an $i_{\ell}(n)$. 
This shows that for all $w$ sufficiently large, we have that $I_w(n)$ is nonempty. Now we iterate this. We return to the situation where $\|u\theta-r_1\|=O(1/u)$. 
For large $w$ start with the minimal $n$ odd such that $q_n\asymp uw$ (and $q_n>w^3$). 
We pick $i$ maximal even such that $q_{n-i}\le q_n/w$. Clearly, $q_{n-i}\asymp q_n/w$. Let $a$ and $b$ be any positive integers which are fixed for the moment. Thus, with $v=aq_{n-i}-bq_{n-i-1}$, we have that 
$$
\|(u+v)\theta-r_1\|=\pm \| u\theta-r_1\|+a\|q_{n-i} \theta\|+b\|q_{n-i-1}\theta\|.
$$
All three players in the right--hand side have sizes $\asymp
w/q_n$. Hence, if we add multiples of $q_n$ of the form $sq_n$ with $s\in [C_{14} w,C_{15}w]$, 
then we will see a change (the fractional part will get to the right of $r_1$ if it was on the left of it and vice versa). 
Now we are ready to prove (v.3) and (v.4).  

To create many of them close together, keep the same conditions, namely $n$ odd, $i$ even minimal such that $q_{n-i}\le q_n/w$, but take $a>b$ between $[w^{1/3}/2,w^{1/3}]$ but otherwise independent. The number of such pairs is $\gg w^{2/3}$. Clearly, with the same choices $v:=aq_{n-i}-bq_{n-i-1}$, we have that 
$$
\|(u+v)\theta-r_1\|\asymp w^{4/3}/q_n,
$$ 
so $m$ can be chosen such that $i_{\ell}=u+v+mq_n$ creates a shift around $r_1$ in the interval 
$[C_{16}w^{4/3} q_n,C_{17} w^{4/3}q_n]$. For large $w$ these $i_{\ell}(n)$'s are all smaller than $w^2q_n$.  And since we have $\gg w^{2/3}$ pairs $(a,b)$, we get that there must be two of them at distance $O(w^{2/3}q_n)=O(i_{\ell}(n)/w^{1/3})$, unless there are some coincidences (so, two pairs $(a_1,b_1)$ and $(a_2,b_2)$ and their corresponding $m$'s will yield the same number). Well, assume they do. So, we have $(a_j,b_j,m_j)$ for $j=1,2$ such that 
$u+a_jq_{n-i}-b_j q_{n-i-1}+m_j q_n$ are the same for $j=1,2$. If the $m_j's$ are not the same, then $|m_1-m_2|q_n\gg q_n$. Since $q_{n-i}\ll q_n/w$ and $a_i,b_i$ are of sizes $O(w^{1/3})$  for $i=1,2$ and also 
$u=O(q_n/w)$, this is impossible for large $w$. So, $m_1=m_2$ and then $(a_1-a_2)q_{n-i}=(b_1-b_2)q_{n-i-1}$. Since $q_{n-i}$ and $q_{n-i-1}$ are coprime, this forces $q_{n-i}$ to divide $b_1-b_2$ and
$q_{n-i-1}$ to divide $a_1-a_2$, which is false for large $w$ (since $q_{n-i}\gg q_n/w>w^2$, while $\max\{|a_1-a_2|,|b_1-b_2|\}\ll w^{1/3}$), unless $a_1=a_2$, $b_1=b_2$, which is not allowed. Thus, these numbers are distinct. Fixing $j$, since we have $\gg w^{2/3}$ values of the $\kappa$'s in an interval of length $O(w^{4/3}q_n)$, we must have two of them whose indices differ by $j$, say $\kappa_{i+j}(n)$ and $\kappa_i(n)$ such that $\kappa_{i+j}(n)-\kappa_i(n)=O(jw^{2/3}q_n)=O(\kappa_{i+j}(n)/w^{2/3})$. So, for a fixed $j$, we can choose the function $f_1(w)\gg w^{2/3}$.  

Finally, for (v.4), let $\kappa_i(n)<\kappa_{i+1}(n)<\cdots$ be consecutive such that for each $j\ge 1$ we have 
$$
\|\kappa_{j+j}(n)\theta-r_{t_j}\|=O\left(\frac{w^{4/3}}{q_n}\right)\qquad {\text{\rm for some}}\quad t_j\in \{1,\ldots,k\}.
$$
As $j$ travels from $1$ to $2k+1$, there will be a repeated value of $t_j$. Say, $t_j=t_{j'}$ for $j<j'$. Then 
$$
\|(\kappa_{i+j'}(n)-\kappa_{i+j}(n))\theta\|=O\left(\frac{w^{4/3}}{q_n}\right).
$$ 
Since $\{a_m\}_{m\ge 0}$ is bounded, the left--hand exceeds $\gg 1/(\kappa_{i+j'}(n)-\kappa_{i+j}(n))$. Thus, we get
$$
\kappa_{i+j'}(n)-\kappa_{i+j}(n)\gg \frac{q_n}{w^{4/3}}\gg \frac{\kappa_{i+j'}(n)}{w^{8/3}},
$$
so $f_2(w)$ can be chosen to be $w^3$. So, (v.3) is satisfied with $L=2k+1$.

\section{Proofs of Theorems}

Theorem \ref{thm:2} is proved. All we need to do is to indicate how Theorem \ref{thm:22} follows. Note that for a fixed $r\in [0,1)$, we have
$$
\sum_{n\ge 1} \frac{1}{b^{\lfloor n\theta+r\rfloor}}=\sum_{m\ge 0} \frac{c_r(m)}{b^m},
$$
where 
$$
c_r(m):=\#\{n: \lfloor n\theta+r\rfloor=m\}.
$$
An easy calculation shows that $c(m)=\lfloor 1/\theta\rfloor+\delta_r(m)$, where 
$$
\delta_r(m):=\left\{\begin{matrix} 0 & {\text{\rm for}} & \left\{(m-r)/\theta\right\}\in [0,1-\{1/\theta\});\\
1 & {\text{\rm for}} & \left\{(m-r)/\theta\right\} \in (1-\left\{1/\theta\right\},1).\end{matrix}\right.
$$
Since $\theta$ is irrational and $r\in (0,1)$, the end points of the above intervals are not achieved. Hence, 
$$
S(b,\theta,A,{\bf v})=T(b,1/\theta,A',{\bf u}),
$$
where $A'$ is the set of values 
$$
1-(1-r_1)/\theta,~(1-(1-r_2)/\theta,\ldots, 1-(1-r_\ell)/\theta \pmod 1.
$$
Note that if $\theta>1$, then the above numbers are already in $(0,1)$ and they are ordered from small to large. If $\theta<1$, the ordering might be different. Also, since condition \eqref{eq:C} is satisfied, it follows that 
the above numbers are indeed incongruent modulo the lattice ${\mathbb Z}+{\mathbb Z}(1/\theta)$. Let us see what the components of ${\bf u}$ are. For the sake of simplicity we only assume that $\theta>1$. We also assume that $v_i\ne 0$ for all $i=1,\ldots,\ell$ (otherwise, if $v_i=0$ then $r_i$ should not be present).  Then $u_0-u_1=v_1,~u_1-u_2=v_2,\ldots, u_{\ell}-u_{\ell+1}=v_{\ell}$ (and $u_{\ell+1}=0$). In particular, 
$u_{j+1}\ne u_j$ for any $j=0,1,\ldots,\ell+1$. Thus, Theorem \ref{thm:22} follows from Theorem \ref{thm:2}.

\section*{Acknowledgements} We thank Yann Bugeaud for pointing out reference \cite{BH}.

\end{document}